\documentclass[11pt]{article}

\usepackage[T1]{fontenc}
\usepackage[utf8]{inputenc}
\usepackage[english]{babel}
\usepackage{amsmath, amssymb, amsthm}
\usepackage[margin=1in]{geometry}
\usepackage[hidelinks]{hyperref}

\theoremstyle{plain}
\newtheorem{theorem}{Theorem}[section]
\newtheorem{lemma}[theorem]{Lemma}
\newtheorem{proposition}[theorem]{Proposition}
\newtheorem{corollary}[theorem]{Corollary}

\theoremstyle{definition}
\newtheorem{definition}[theorem]{Definition}

\newtheorem{remark}[theorem]{Remark}

\newcommand{\Sn}{\mathfrak{S}_n}
\newcommand{\GS}{G_S}

\title{Unranking $S$-avoiding permutations with two forbidden offsets}
\author{Antônio Augusto André\\\small Inteli --- Instituto de Tecnologia e Liderança, São Paulo, Brazil}
\date{}

\begin{document}
\maketitle

\begin{abstract}
Given a finite family of combinatorial objects with a total order, ranking asks for the position of a given object, and unranking asks for the object occupying a given position. Both provide algorithmic access to families too large to enumerate. Kagey gave a polynomial-time algorithm for several position-restricted families and left open the case of $\{-1,\,1\}$-avoiding permutations. We show that when $|S| = 2$, the graph of forbidden squares is $2$-regular, a disjoint union of cycles, and that fixing a prefix leaves only paths and untouched cycles. This yields polynomial-time prefix counting, ranking and unranking for every two-element set of offsets, covering $\{-1,\,1\}$. The same structure gives a short proof that the number of $\{-1,\,1\}$-avoiding permutations equals the $n$-th ménage number for odd $n$ and exceeds it by $2$ for even $n$, proving a conjecture recorded in the OEIS entry A354152.
\end{abstract}

\section{Introduction}

Ranking a finite family of combinatorial objects under a total order means computing the position of a given object; unranking means producing the object at a given position. In 2024, Kagey~\cite{kagey} gave polynomial-time algorithms for both, for several families of position-restricted permutations, including derangements and ménage permutations, and observed that the same method extends to $\{i\}$-avoiding and $\{i,\,i+1\}$-avoiding permutations.

The obstruction described in \cite{kagey} is geometric: after placing a rook and deleting a row and a column, the two sides of the remaining board need not separate into the staircase pieces exploited in the ménage case. The observation of this note is that geometric separation is not necessary. The relevant structure is the graph of forbidden squares. For $|S|= 2$, this graph is $2$-regular, therefore a disjoint union of cycles. Fixing a prefix deletes vertices, and the remaining components are the arcs between them. In the column order they are interleaved, in the cycle they are intervals.

Counting $S$-avoiding permutations for a fixed two-element set $S$ is classical: the corresponding boards are circulant, and the rook polynomials of sums of powers of the cyclic permutation matrix have been studied by Kamenetsky~\cite{kamenetsky}. For $S = \{-1,1\}$, entry A354152 of the On-Line Encyclopedia of Integer Sequences (OEIS)~\cite{oeis} records a recurrence obtained by Zeilberger's method~\cite{zeilberger}. The problem here is different: counting the permutations that extend a given prefix. The structure developed for that purpose also settles a counting question, and Section~\ref{sec:application} records what it gives for $S = \{-1,1\}$.

\section{Preliminaries}

Throughout, $n$ denotes a positive integer, $[n] = \{1, \dots, n\}$, and $\Sn$
is the group of permutations of $[n]$.

\begin{definition}
Let $S \subset \mathbb{Z}$ be finite. A permutation $\pi \in \Sn$ is
\emph{$S$-avoiding} if $\pi(i) - i - s \not\equiv 0 \pmod n$ for every
$i \in [n]$ and every $s \in S$.
\end{definition}

We encode the restriction as a board, following \cite{kaplanskyriordan}. The
board $B_S$ consists of the pairs $(i,j) \in [n] \times [n]$ with
$j \not\equiv i+s \pmod n$ for every $s \in S$, and its complementary board
$B^c_S$ consists of the remaining pairs, the forbidden squares. The
$S$-avoiding permutations are exactly the placements of $n$ pairwise
nonattacking rooks on $B_S$.

\begin{definition}
Let $B$ be a board. Its rook polynomial is $p_B(x) = \sum_{k \geq 0} r_k x^k$
where $r_k$ is the number of ways of placing $k$ nonattacking rooks on the
squares of $B$.
\end{definition}

For a board $B$, let $G(B)$ be the graph with vertex set $B$ and an edge
between two squares whenever they share a row or a column.

\begin{definition}
A \emph{prefix} of length $\ell$, with $0 \leq \ell \leq n$, is a tuple
$\alpha = (\alpha_1, \dots, \alpha_\ell)$ of distinct elements of $[n]$;
for $\ell = 0$ it is the empty tuple. It is \emph{valid} if
$\alpha_i - i - s \not\equiv 0 \pmod n$ for every $i \leq \ell$ and every
$s \in S$. A permutation $\pi \in \Sn$ \emph{extends} $\alpha$ if
$\pi(i) = \alpha_i$ for every $i \leq \ell$, and $\#\mathrm{prefix}(\alpha)$
denotes the number of $S$-avoiding permutations extending $\alpha$.
\end{definition}

\begin{definition}
Let $\alpha$ be a prefix of length $\ell$. Deleting the rows $1, \dots, \ell$
and the columns $\alpha_1, \dots, \alpha_\ell$ leaves $n - \ell$ rows and
$n - \ell$ columns; relabel each of these two sets by $[n-\ell]$ in increasing
order. The \emph{derived board} $B_\alpha \subseteq [n-\ell] \times [n-\ell]$
consists of the squares of $B_S$ lying in a surviving row and a surviving
column, and $B^c_\alpha$ consists of those of $B^c_S$. The relabelling is a
bijection preserving rows and columns, so it preserves nonattacking placements;
in particular $B_\alpha$ and $B^c_\alpha$ are complementary boards in
$[n-\ell] \times [n-\ell]$.
\end{definition}

\begin{theorem}[{\cite[Theorem 9]{kagey}}]
Let $\mathcal{W}$ be a finite set of words ordered lexicographically over the
alphabet $[n]$. If the number of words of $\mathcal{W}$ beginning with a given
prefix can be computed in polynomial time, ranking and unranking in
$\mathcal{W}$ can be computed in polynomial time.
\end{theorem}

We take $\mathcal{W}$ to be the set of $S$-avoiding permutations, with each $\pi$
regarded as the word $(\pi(1)\dots\pi(n))$.

\section{The forbidden board is a union of cycles}
\begin{definition}
Let $S = \{s_1, s_2\}$ with $s_1 \not\equiv s_2 \pmod n$. Write $\GS := G(B^c_S)$, the graph of forbidden squares of $S$.
\end{definition}

\begin{lemma}\label{lem:2regular}
$\GS$ is a disjoint union of cycles.
\end{lemma}

\begin{proof}
Write $A_i = (i, i+s_1)$ and $B_i = (i, i+s_2)$ for $i \in [n]$, with column indices
reduced modulo $n$. Since $s_1 \not\equiv s_2$, there are $2n$ distinct squares; these are the forbidden squares.

Row $i$ contains the two forbidden squares $A_i$ and $B_i$, so each of them has exactly one neighbor from a shared row. Column $j$ contains the two forbidden squares $A_{j-s_1}$ and $B_{j-s_2}$, so each square has exactly one neighbor from a shared column. These neighbors are distinct since a common row and column would force $s_1 \equiv s_2$.

Hence every vertex of $\GS$ has degree exactly $2$. A finite graph where each vertex has degree exactly $2$ is a disjoint union of cycles.
\end{proof}

\begin{lemma}\label{lem:gcd}
$\GS$ has exactly $\gcd(|s_1-s_2|,\, n)$ cycles, each of length $2n / \gcd(|s_1-s_2|,\, n)$.
\end{lemma}

\begin{proof}
Write $d = s_1 - s_2$ and traverse a cycle of $\GS$. From $A_i$, the neighbor arising from a shared row is $B_i$. From $B_i = (i, i+s_2)$, the neighbor arising from a shared column is $A_j$, with $A_j = (j, j+s_1)$. Since they share the same column, $i+s_2 = j+s_1$ and $A_j = A_{i-d}$.

The indices appearing in the cycle of $A_i$ are the orbit of $i$ under the translation $i \mapsto i-d$. The orbit closes at the least $k > 0$ with $kd \equiv 0 \pmod n$, so $k = n/ \gcd(|d|,\,n)$ and corresponds to $n/ \gcd(|d|,\,  n)$ elements. Since each element has 2 vertices of $\GS$ that correspond to $A_i$ and $B_i$, the length of a cycle is $2n/ \gcd(|d|, \, n)$. Since the orbits partition the $n$ indices, there are $\gcd(|d|,\,n)$ cycles.
\end{proof}

\section{The board induced by a prefix}

\begin{lemma}\label{lem:rowcolpair}
Let $i, j \in [n]$. The two squares of $\GS$ lying in row $i$ are adjacent in
$\GS$, and so are the two squares lying in column $j$.
\end{lemma}

\begin{proof}
Row $i$ contains exactly the squares $A_i = (i,\, i + s_1)$ and
$B_i = (i,\, i + s_2)$, and these are distinct because $s_1 \not\equiv s_2
\pmod n$. They share row $i$, so by definition they are adjacent in $\GS$.

Column $j$ contains exactly the squares $A_{j - s_1}$ and $B_{j - s_2}$, which
are again distinct, and they share column $j$; hence they are adjacent as well.
\end{proof}

\begin{remark}\label{rmk:overlap}
A prefix $\alpha$ of length $\ell$ deletes rows $1, \dots, \ell$ and columns
$\alpha_1, \dots, \alpha_\ell$. By Lemma~\ref{lem:rowcolpair}, each such deletion
removes the two endpoints of an edge of $\GS$. These edges need not be disjoint:
a square may lie in a deleted row and in a deleted column at once. The set of
deleted vertices is therefore a union of edges of $\GS$, not necessarily a
disjoint one.
\end{remark}

\begin{theorem}\label{thm:components}
Let $B^c_\alpha$ be the complementary board induced by a valid prefix $\alpha$.
Every connected component of $G(B^c_\alpha)$ is either an untouched cycle of
$\GS$ or a path.
\end{theorem}

\begin{proof}
The board $B^c_\alpha$ consists of the squares of $\GS$ that lie in no deleted
row and in no deleted column. Two surviving squares share a row or a column in
$B^c_\alpha$ exactly when they do so in $\GS$, so $G(B^c_\alpha)$ is the subgraph
of $\GS$ induced by the surviving vertices.

By Lemma~\ref{lem:2regular}, $\GS$ is a disjoint union of cycles, and an induced
subgraph of a disjoint union of cycles is the disjoint union of the induced
subgraphs of each cycle. It therefore suffices to consider a single cycle $C$.

If no vertex of $C$ is deleted, the induced subgraph is $C$ itself. Otherwise,
deleting a nonempty set of vertices from a cycle leaves a disjoint union of
paths, one for each maximal run of consecutive surviving vertices. Hence every
component of $G(B^c_\alpha)$ is an untouched cycle or a path.
\end{proof}

\section{Prefix counting and unranking}
\begin{proposition}\label{prop:polys}
Let $P_m$ and $Q_m$ denote the rook polynomials of a path and of a cycle with $m$ squares, respectively. Then:
\[
  P_m(x) = \sum_{k=0}^{\lceil m/2 \rceil} \binom{m-k+1}{k} x^k,
  \qquad
  Q_m(x) = \sum_{k=0}^{\lfloor m/2 \rfloor} \frac{m}{m-k} \binom{m-k}{k} x^k.
\]
\end{proposition}

\begin{proof}
Two squares of a path or of a cycle share a row or a column precisely when they are consecutive, so a set of $k$ squares is nonattacking when no two elements are consecutive. Choosing $k$ nonconsecutive elements from a path with $m$ squares can be done in $\binom{m-k+1}{k}$ ways, while choosing $k$ nonconsecutive elements from a cycle with $m$ squares can be done in $\frac{m}{m-k} \binom{m-k}{k}$ ways; these are Kaplansky's two lemmas~\cite{kaplansky}. The first polynomial is equal to the $m$-th Fibonacci polynomial from \cite[Definition 29]{kagey}, which by \cite[Lemma 30]{kagey} is the rook polynomial of the staircase-shaped board with $m$ squares.
\end{proof}

\begin{corollary}\label{cor:product}
Let $\mathcal{C}_\alpha$ denote the set of connected components of $G(B^c_\alpha)$. Then:
\[
p_{B^c_\alpha}(x) = \prod_{K \in \mathcal{C}_\alpha} p_K(x)
\]
where $p_K = P_{|K|}$ if $K$ is a path and $p_K = Q_{|K|}$ if $K$ is a cycle. In particular, $p_{B^c_\alpha}$ is determined by the sizes of the components together with their types.
\end{corollary}

\begin{proof}
By Theorem~\ref{thm:components}, every component is a path or a cycle. Two squares lying in distinct components share neither a row nor a column, since otherwise they would be adjacent in $G(B^c_\alpha)$ and would lie in the same component. The components are pairwise disjoint boards in the sense of \cite[Definition 25]{kagey} and \cite[Theorem 26]{kagey} gives that the rook polynomial of $B^c_\alpha$ is the product of the rook polynomials of its components.
\end{proof}

\begin{theorem}\label{thm:main}
For every $|S| = 2$ and $s_1 \not\equiv s_2$ there is a polynomial algorithm computing $\#\mathrm{prefix}(\alpha)$ for any valid prefix and a polynomial lexicographic ranking and unranking algorithm for $S$-avoiding permutations.
\end{theorem}

\begin{proof}
The graph $\GS$ has exactly $2n$ vertices, so the cycles and their traversal orders can be computed in $O(n)$.

Let $\alpha$ be a valid prefix of length $\ell$, and let $D$ be the set of positions of the deleted squares along a cycle, with different sets for different cycles, that is, the union, over all deleted rows $i$, of $\{A_i, B_i\}$ and all deleted columns $j$ of $\{A_{j-s_1}, B_{j-s_2}\}$. By Lemma~\ref{lem:rowcolpair}, each of these contributions is a pair of adjacent positions; by Remark~\ref{rmk:overlap}, the pairs need not be disjoint, so $|D|$ may be smaller than $4\ell$ but this does not affect the conclusion of the argument below.

By Theorem~\ref{thm:components}, a cycle meeting no element of $D$ survives as a single component. On a cycle meeting $D$, the components are paths whose lengths are the gaps between consecutive elements of $D$ in the cyclic order, including the gap from the last element back to the first. These lengths are obtained in time $O(n)$ by sorting $D$ along each cycle.

By Corollary~\ref{cor:product}, the rook polynomial $p_{B^c_\alpha}$ is the product of the corresponding factors, which is computed in time $O(n^2)$ by naive convolution, and \cite[Lemma 21]{kagey}, applied to the derived board $B_\alpha$, then yields
\begin{equation}\label{eq:ie}
\#\mathrm{prefix}(\alpha)
  \;=\; \sum_{k \geq 0} (-1)^k\, r^c_k\, (n - \ell - k)!,
\end{equation}
where $r^c_k$ denotes the coefficient $r_k$ of $p_{B^c_\alpha}$. Hence $\#\mathrm{prefix}(\alpha)$ can be computed in polynomial time and \cite[Theorem 9]{kagey} gives a polynomial-time lexicographic ranking and unranking algorithm.

The unranking solution needs to look at every prefix size from $1$ to $n$ and every possible value $1$ to $n$ that can be placed on every prefix, so at most $n^2$ prefix counts are needed. Since each prefix is computed in time $O(n^2)$, it gives an overall bound of $O(n^4)$ arithmetic operations on integers with $O(n \log n)$ bits.
\end{proof}

\begin{corollary}\label{cor:pm11}
There is a polynomial-time lexicographic ranking and unranking algorithm for $\{-1,1\}$-avoiding permutations.
\end{corollary}

\begin{proof}
For $n > 2$ we have that $-1 \not\equiv 1\pmod n$ so the previous theorem applies for $S = \{-1,1\}$. The remaining cases $n\leq2$ are trivial.
\end{proof}

\begin{remark}
Corollary~\ref{cor:pm11} answers~\cite[Problem 41]{kagey}. The values $-1$ and $1$ play no role in the argument, so Theorem~\ref{thm:main} settles \cite[Open Question 40]{kagey} for every $S$ with $|S| \leq 2$.
\end{remark}

\section{An application: the case $S = \{-1,1\}$}\label{sec:application}

By Lemma~\ref{lem:gcd} the shape of the board depends on $S$ only through $\gcd(|s_1-s_2|,\,n)$. Two different sets of offsets may therefore produce the same multiset of cycle lengths, or multisets that differ in a controlled way; in both cases their counts can be compared. We carry this out for $S = \{-1,1\}$ against the m\'enage case, and the comparison settles a conjecture recorded in the OEIS.

Write $a_S(n)$ for the number of $S$-avoiding permutations of $[n]$, that is, for the number of extensions of the empty prefix. Taking $\ell = 0$ in \eqref{eq:ie},
\begin{equation}\label{eq:ie0}
  a_S(n) \;=\; \sum_{k \geq 0} (-1)^k \, r_k \, (n-k)! ,
\end{equation}
where $r_k$ is the number of ways of choosing $k$ pairwise nonattacking forbidden squares, so that $r_0 = 1$ and $r_k = 0$ for $k > n$. By Corollary~\ref{cor:product}, applied with $\ell = 0$, the numbers $r_k$ are the coefficients of the product of the rook polynomials of the components of the full forbidden board.

\begin{lemma}\label{lem:lucas}
Let $u$ and $v$ be the two roots of $t^2 - t - x$ in an algebraic closure of $\mathbb{Q}(x)$, so that $u + v = 1$ and $uv = -x$. Then $Q_m(x) = u^m + v^m$ for every $m \geq 1$, and consequently
\[
  Q_n(x)^2 \;=\; Q_{2n}(x) \;+\; 2(-x)^n \qquad (n \geq 1).
\]
\end{lemma}

\begin{proof}
The Girard--Waring formula~\cite{waring} expresses the power sum of two quantities in terms of their elementary symmetric functions $e_1 = u + v$ and $e_2 = uv$ as
\[
  u^m + v^m \;=\; \sum_{k = 0}^{\lfloor m/2 \rfloor} \frac{m}{m-k}\binom{m-k}{k}\,(-1)^k\, e_1^{\,m-2k}\, e_2^{\,k} .
\]
Substituting $e_1 = 1$ and $e_2 = -x$ turns $(-1)^k e_2^{\,k}$ into $x^k$ and leaves the closed formula for $Q_m$ of Proposition~\ref{prop:polys}. Squaring,
\[
  Q_n(x)^2 = (u^n + v^n)^2
           = u^{2n} + v^{2n} + 2(uv)^n
           = Q_{2n}(x) + 2(-x)^n . \qedhere
\]
\end{proof}

By \cite[Definition 24]{kagey} a m\'enage permutation is a permutation $\pi \in \Sn$ with $\pi(i) \neq i$ and $\pi(i) + 1 \not\equiv i \pmod n$ for every $i \in [n]$; equivalently, by \cite[Example 39]{kagey}, the m\'enage permutations are the $\{-1,0\}$-avoiding ones. Write $U(n) = a_{\{-1,0\}}(n)$ for their number, the sequence A000179 of~\cite{oeismenage}. Here $|s_1 - s_2| = 1$, so by Lemma~\ref{lem:gcd} the m\'enage board is a single cycle with $2n$ squares and its rook polynomial is $Q_{2n}$.

\begin{proposition}\label{prop:odd}
For every odd $n \geq 3$, $\;a_{\{-1,1\}}(n) = U(n)$.
\end{proposition}

\begin{proof}
For $S = \{-1,1\}$ we have $|s_1 - s_2| = 2$ and $\gcd(2,n) = 1$ because $n$ is odd, so by Lemma~\ref{lem:gcd} this board is also a single cycle with $2n$ squares. By Proposition~\ref{prop:polys}, the two boards therefore have the same rook polynomial $Q_{2n}$, hence the same numbers $r_k$. The right-hand side of \eqref{eq:ie0} depends on the board only through these numbers, so the two counts agree.
\end{proof}

\begin{theorem}\label{thm:even}
For every even $n \geq 4$, $\;a_{\{-1,1\}}(n) = U(n) + 2$.
\end{theorem}

\begin{proof}
Now $\gcd(2,n) = 2$, so by Lemma~\ref{lem:gcd} the board of $S = \{-1,1\}$ is a disjoint union of two cycles with $n$ squares each, and by Corollary~\ref{cor:product} its rook polynomial is $Q_n(x)^2$. By Lemma~\ref{lem:lucas}, and because $n$ is even,
\[
  Q_n(x)^2 - Q_{2n}(x) \;=\; 2(-x)^n \;=\; 2x^n .
\]
The difference is a single monomial, so the two boards have the same $r_k$ for every $k \neq n$, while $r_n$ for $S = \{-1,1\}$ exceeds the ménage value by exactly $2$. Subtracting the two instances of \eqref{eq:ie0} therefore leaves a single term,
\[
  a_{\{-1,1\}}(n) - U(n) \;=\; (-1)^n \cdot 2 \cdot (n-n)! \;=\; 2 . \qedhere
\]
\end{proof}

\begin{remark}
At $x = 1$ the polynomials $P_m$ and $Q_m$ specialize to the Fibonacci and the Lucas numbers, and Lemma~\ref{lem:lucas} specializes to the classical identity $L_n^2 - L_{2n} = 2(-1)^n$. What Lemma~\ref{lem:gcd} adds is the dictionary between a set of offsets and a multiset of cycles, which turns an identity between polynomials into an identity between two counting problems. The entry A354152~\cite{oeis} records the relation of Proposition~\ref{prop:odd} without proof and states that of Theorem~\ref{thm:even} as a conjecture; the two results prove both.
\end{remark}

\begin{remark}
The argument uses the $2$-regularity of the graph essentially, so $|S| \geq 3$ is not covered. 
For $|S| = k$, each forbidden square has $k-1$ neighbors in its row and $k-1$ neighbors in its column, giving $2(k-1)$ neighbors, hence for $k \geq 3$ Lemma~\ref{lem:2regular} no longer applies, and \cite[Open Question 40]{kagey} remains open for those sets.
\end{remark}

\section*{Acknowledgments}
I thank Peter Kagey for the correspondence on Problem 41.

The identity of Lemma~\ref{lem:lucas} and its application in Section 6 were obtained by a large language model, which identified the formula of Proposition~\ref{prop:polys} as a Girard--Waring power sum and suggested comparing the board of $S = \{-1,\,1\}$ with the ménage board. The results of Sections 3 to 5 are the author's own. All statements and proofs were verified by the author.

\end{document}